\documentclass[a4paper,pdftex,reqno,10pt]{amsart}

\usepackage{amsmath}
\usepackage{amssymb}
\usepackage{url}
\usepackage{algorithm}
\usepackage{algpseudocode}
\usepackage{caption}
\usepackage{varwidth}
\usepackage{makecell}
\usepackage[colorlinks]{hyperref}

\newtheorem{theorem}{Theorem}
\newtheorem{lemma}{Lemma}
\newtheorem{corollary}{Corollary}
\newtheorem{example}[theorem]{Example}
\newtheorem{definition}[theorem]{Definition}

\newcommand{\F}[2]{\mathbb{F}_{#2}^{#1}}
\newcommand{\EF}[2]{\operatorname{EF}_{#2}(#1)}
\newcommand{\PG}[2]{\mathcal{P}_{#2}(#1)}
\newcommand{\G}[3]{\mathcal{G}_{#3}(#1,#2)}
\newcommand{\gaussm}[3]{\genfrac{[}{]}{0pt}{}{#1}{#2}_{#3}}
\newcommand{\rk}{\operatorname{rk}}

\usepackage{color}
\ifdefined\myred
\newcommand{\highlight}[1]{\textcolor{red}{#1}}
\else
\newcommand{\highlight}[1]{#1}
\fi

\ifdefined\mygreen
\newcommand{\highlightgreen}[1]{\textcolor{green}{#1}}
\else
\newcommand{\highlightgreen}[1]{#1}
\fi

\begin{document}
\title{Coset Construction for Subspace Codes}

\author{Daniel~Heinlein and~Sascha~Kurz}
\email{\href{mailto:daniel.heinlein@uni-bayreuth.de}{daniel.heinlein@uni-bayreuth.de}}
\email{\href{mailto:sascha.kurz@uni-bayreuth.de}{sascha.kurz@uni-bayreuth.de}}

\thanks{The work of the authors was supported by the grant KU~2430/3-1 and WA~1666/9-1 ``Integer Linear Programming Models for Subspace Codes and Finite Geometry'' from the German Research Foundation and the COST~Action~IC1104 ``Random Network
Coding and Designs over $\operatorname{GF}(q)$''.}


\begin{abstract}
One of the main problems of \highlightgreen{the research} area of network coding is to compute good lower
and upper bounds of the achievable \highlight{cardinality of} so-called subspace codes in $\mathbf{\PG{n}{q}}$\highlightgreen{, i.e., the set of subspaces of $\mathbb{F}_q^n$,}
for a given
minimal distance. Here we generalize a construction of Etzion and Silberstein to a wide range
of parameters. This construction, named \emph{coset construction}, improves \highlightgreen{or attains} several of the
previously \highlight{best-known} subspace code \highlightgreen{sizes} and attains the MRD bound for an infinite family of parameters.
\end{abstract}

\maketitle

\section{Introduction}
Let $\mathbb{F}_q$ be the finite field of order $q$ and $V$ be a vector space
of dimension $n$ over $\mathbb{F}_q$. Since $V$ is isomorphic to $\F{n}{q}$, we will assume
$V=\F{n}{q}$ in the following. By $\G{n}{k}{q}$ we denote the set of all $k$-dimensional
subspaces of $\F{n}{q}$, where $0\le k\le n$. The \emph{projective space} of order $n$ over
$\mathbb{F}_q$ is given by $\PG{n}{q}=\cup_{0\le k\le n} \G{n}{k}{q}$. It is well known
that
\[
  d_S(U,W):=\dim U+\dim W-2\dim(U\cap W)
\]
is a metric on $\PG{n}{q}$ \cite{ahlswede2001perfect}. Thus, one can define codes on
$\PG{n}{q}$ and $\G{n}{k}{q}$, which are called \emph{subspace codes} and \emph{constant
dimension codes}, respectively. \highlight{The distance function $d_S$ is known as
\emph{subspace distance} and one of the two distance functions that can be motivated by an
information-theoretic analysis of the so-called \highlightgreen{Silva-Kschischang-K{\"o}tter} model \cite{silva2008rank}.
The second distance function is the so-called \emph{injection distance}
$d_I(U,V):=\max\left\{\dim U,\dim V\right\}-\dim(U\cap V)$. For two subspaces of the same
dimension we have $d_S(U,V)=2d_I(U,V)$, i.e., the two metrics
are equivalent on $\G{n}{k}{q}$, and
$d_I(U,V)\le d_S(U,V)\le 2 d_I(U,V)$ in general.} We say that $\mathcal{C}\subseteq \PG{n}{q}$ is an $(n,M,d)_q$
code \highlightgreen{in the projective space} if $\left|\mathcal{C}\right|=M$ and $d(U,V)\ge d$ for all $U,V\in\mathcal{C}$.
If $\mathcal{C}\subseteq \G{n}{k}{q}$ for some $k$, we speak of an $(n,M,d;k)_q$ code. The \emph{minimum distance}
of a code $\mathcal{C}\subseteq \PG{n}{q}$ is denoted by $D_S(\mathcal{C}):=\min_{U\neq V\in\mathcal{C}} d_S(U,V)$.
One major problem is the determination of the maximum size $A_q(n,d)$ of an $(n,M,d)_{\highlight{q}}$ code in $\PG{n}{q}$ and the maximum
size $A_q(n,d;k)$ of an $(n,M,d;k)_{\highlight{q}}$ code in $\G{n}{k}{q}$. Bounds for $A_q(n,d)$ and $A_q(n,d;k)$ \highlight{have been heavily studied},
see e.g.\ the survey \cite{etzion2015galois} or the new \highlight{online} database
at \url{http://subspacecodes.uni-bayreuth.de} \cite{TableSubspacecodes}.

With respect to lower bounds on $A_q(n,d;k)$, an asymptotically optimal
construction is given by lifted maximum-rank-distance
codes \cite{gabidulin1985theory,silva2008rank}.
\highlightgreen{To be more precise, the
rate of \highlight{the} transmission $\frac{\log_q \left|\mathcal{C}\right|}{n\cdot \max_{U\in\mathcal{C}}\dim(U)}$
is asymptotically optimal \highlight{up to a constant factor} \cite{koetter2008coding}. A rough estimation between $\left|\mathcal{C}\right|$
and the Singleton bound yields an approximation factor of at most~$4$.}
The concept of maximum-rank-distance codes was generalized
from \highlight{arbitrary} rectangular matrices to matrices with a (structured) set of prescribed zeros in
\cite{etzion2009error} and used to combine several maximum-rank-distance codes to \highlight{generate} a constant dimension code
-- the so-called multilevel or Echelon-Ferrers construction. Many of the \highlight{best-known} lower bounds on $A_q(n,d;k)$
arise from this construction. However, it is rather general and involves several search spaces or optimization
problems in order to be evaluated optimally. For special subclasses explicit variants of the construction and indeed
explicit formulas for the sizes of the corresponding codes have been obtained, see \cite{skachek2010recursive}.
We remark that additional refinements of the Echelon-Ferrers construction have been proposed
recently, see \highlight{\cite{MR3480069,etzion2013codes,silberstein2014subspace}}.

\highlight{An improvement beyond the Echelon-Ferrers construction was
Construction~III in \cite{etzion2013codes} giving $A_2(8,4;4)\ge 4797$.
The authors conjecture
that the underlying idea can be generalized to further parameters assuming the existence
of a corresponding parallelism. In Theorem~\ref{thm_coset_parallelism} we will show that this is indeed the case. Moreover,
there is a more general underlying construction for $(n,M,d;k)_q$ and $(n,M,d)_q$ codes that is capable of improving some
of the so far best-known lower bounds on $A_q(n,d;k)$, which is the core of this paper. To this end, we will
give several infinite, parametric families of constructions as well as sporadic examples.}

The remaining part of the paper is organized as follows. In Section~\ref{sec_preliminaries} we collect
some facts about representations of subspaces, MRD codes, parallelisms, and the Echelon-Ferrers construction. The
main idea of the coset construction is described in Section~\ref{sec_coset_construction}. Since this construction
has several degrees of freedom, we present some first insights on the choice of {\lq\lq}good{\rq\rq} parameters
in Section~\ref{sec_optimal_parameters}. After listing some examples improving \highlightgreen{or attaining} several lower bounds on $A_q(n,d;k)$
in Section~\ref{sec_examples}\highlight{, we conclude with open questions} in Section~\ref{sec_conclusion}.

\section{Preliminaries}
\label{sec_preliminaries}
In this section we summarize some notation and \highlight{well-known} insights that will be used in the
 later parts of the paper.

\subsection{Gaussian elimination and representations of subspaces}
\label{subsec_gaussian_elimination}
Let $A\in\mathbb{F}_q^{k\times n}$ be a matrix of (full) rank $k$. The row-space of $A$ forms \highlight{a}
$k$-dimensional subspace of $\F{n}{q}$. The matrix $A$ is called \emph{generator matrix} of a given
element of $\G{n}{k}{q}$. Since the  application of the Gaussian elimination algorithm \highlight{on} a generator matrix
$A$ does not change the row-space, we can restrict ourselves \highlight{on} generator matrices which are
in \emph{reduced row echelon form} (\highlightgreen{RRE} \highlight{form}), i.e., the matrix has the shape resulting from Gaussian elimination.
The representation is unique and does not depend on the elimination algorithm.
This well-known connection is indeed a bijection, which we denote by
\begin{align}
\highlightgreen{
\tau:\G{n}{k}{q}\rightarrow \left\{A'\in \mathbb{F}_q^{k\times n}:\operatorname{rk}(A')=k, A'\text{ in \highlightgreen{RRE} form}\right\}.
}
\label{eqn:tau}
\end{align}
This observation
is capable \highlight{of easily explaining} many properties of $\G{n}{k}{q}$ so that we commonly identify the elements of $\G{n}{k}{q}$
with their corresponding generator matrices in reduced row echelon  form.

Given a matrix $A\in \mathbb{F}_q^{k\times n}$ of full rank we denote by $p(A)\in\mathbb{F}_2^n$ the
binary vector whose $1$-entries coincide with the pivot columns of $A$. For each $v\in\mathbb{F}_2^n$
let $\EF{v}{q}$ denote the set of all $k\times n$ matrices over $\mathbb{F}_q$
that are in reduced row echelon form with pivot columns described by $v$, where $k$ is the weight of $v$.

\begin{example}
  For $v=(1,0,1,1,0)$
  we have
  \[
    \EF{v}{q}=\left\{\begin{pmatrix}
    1 & \star & 0 & 0 & \star \\
    0 & 0     & 1 & 0 & \star \\
    0 & 0     & 0 & 1 & \star
    \end{pmatrix}\right\},
  \]
  where the $\star$s represent arbitrary elements of $\mathbb{F}_q$, i.e.,
  $\left|\EF{v}{q}\right|=q^4$.
\end{example}

In general we have
\[
  \left|\operatorname{EF}_q\Big((v_1,\dots,v_n)\Big)\right|= q^{\sum\limits_{i=1}^{n}\left(1-v_i\right)\cdot \sum\limits_{j=1}^{i} v_j}\highlightgreen{.}
\]
and the structure of the corresponding matrices can be read off from the corresponding
\emph{(Echelon)-Ferrers diagram}
\[
  \begin{array}{ll}
    \bullet & \bullet \\
            & \bullet \\
            & \bullet
  \end{array},
\]
where the pivot columns and zeros are omitted and the stars are replaced by solid black circles.
\highlightgreen{A Ferrers diagram represents partitions as patterns of dots, with the
\highlight{$i$}th row having the same number of dots as the \highlight{$i$}th term \highlight{$s_i$} in the partition $n'=s_1+\dots+s_l$, where
$s_1\ge\dots\ge s_l$ and \highlight{$s_j\in\mathbb{N}_{>0}$}, \highlight{cf.}~\cite{andrews1998theory}. Usually
a Ferrers diagram is depicted in such a way that it is the vertically mirrored version of the above
constructed (Echelon)-Ferrers diagram. \highlight{In the special case of Echelon-Ferrers diagrams, we have $n'=\sum_{i=1}^n(1-v_i)\cdot \sum_{j=1}^i v_j$.}}

By summing over all binary vectors of weight $k$ in $\F{n}{2}$ one can compute
\[
  \left|\G{n}{k}{q}\right|=\gaussm{n}{k}{q}:=
  \prod_{i=1}^{k} \frac{q^{n-k+i}-1}{q^i-1},
\]
where $\gaussm{n}{k}{q}$ is called \emph{Gaussian binomial coefficient}.

Later on we will use the inverse operation of deleting the pivot columns of a
matrix in \highlightgreen{RRE} form:
\begin{definition}
  Let $B\in\mathbb{F}_q^{k\times n}$ be a full-rank matrix in \highlightgreen{RRE} form and $F\in\mathbb{F}_q^{k'\times (n-k)}$ be arbitrary,
  where $k,k',n\in \mathbb{N}$ and $k\le n$. Let further $f^i$ denote the $i$th column of $F$ \highlightgreen{for $i \in \{1,\ldots,n\}$}. Then,
  $G=\varphi_{B}(F)$ denotes the $k'\times n$ matrix over $\mathbb{F}_q$ whose columns are given by
  $g^i=\mathbf{0}\in\mathbb{F}_q^{k'}$ if $v_i=1$ and \highlight{$g^i=f^{i-s_i}$} otherwise, where $(v_1,\dots,v_n)=p(B)$ and $s_i=\sum_{j=1}^i v_j$,
  for all $1\le i\le n$.
\end{definition}

\begin{example}
  For
  \[
    B=\begin{pmatrix}
      0 & 1 & 1 & 0 & 1 & 0  \\
      0 & 0 & 0 & 1 & 0 & 0 \\
      0 & 0 & 0 & 0 & 0 & 1 \\
    \end{pmatrix}
    \text{ and }F=\begin{pmatrix}
      1 & 0 & 0 \\
      0 & 1 & 0 \\
      1 & 0 & 0 \\
      0 & 1 & 0 \\
    \end{pmatrix}
  \]
  we have $p(B)=(0,1,0,1,0,1)$, \highlightgreen{$s_1=0, s_2=s_3=1, s_4=s_5=2, s_6=3$,} and
  \[
    \varphi_B(F)=\begin{pmatrix}
      1 & 0 & 0 & 0 & 0 & 0 \\
      0 & 0 & 1 & 0 & 0 & 0 \\
      1 & 0 & 0 & 0 & 0 & 0 \\
      0 & 0 & 1 & 0 & 0 & 0 \\
    \end{pmatrix}.
  \]
\end{example}

\subsection{MRD codes and the Echelon-Ferrers construction}
For matrices $A,B\in\mathbb{F}_q^{m\times n}$ the \emph{rank distance}
is defined via $d_R(A,B):=\operatorname{rk}(A-B)$.
It is indeed a metric, as observed in \cite{gabidulin1985theory}.
The maximum possible cardinality of a rank-metric
code with given minimum rank distance is exactly determined in
all cases.
\begin{theorem}(see \cite{gabidulin1985theory})
  \label{thm_MRD_size}
  Let $m,n\ge d$ be positive integers, $q$ a prime power, and $\mathcal{C}\subseteq \mathbb{F}_q^{m\times n}$ be a rank-metric
  code with minimum rank distance $d$. Then, $|\mathcal{C}|\le q^{\max(n,m)\cdot (\min(n,m)-d+1)}$.
  Codes attaining this upper bound are called maximum-rank distance (MRD) codes. They exist for all (suitable) choices of parameters.
\end{theorem}
If $m<d$ or $n<d$, then only $|\mathcal{C}|=1$ is possible, which \highlight{can be combined to give a} single upper bound
$|\mathcal{C}|\le \left\lceil q^{\max(n,m)\cdot (\min(n,m)-d+1)}\right\rceil$.
Using an $m\times m$ identity matrix as a prefix\highlight{, one obtains the corresponding subspace codes known as} lifted MRD codes.

\begin{theorem}(see \cite{silva2008rank})
  \label{thm_lifted_MRD}
  For positive integers $k,d,n$ with $k\le n$, $d\le 2\min(k,n-k)$, and $d\equiv 0\pmod{2}$, the size of a lifted MRD code in \highlight{$\G{n}{k}{q}$} with
  subspace distance $d$ is given by
  \[
    M(q,k,n,d):=
    q^{\max\highlight{\{}k,n-k\highlight{\}}\cdot(\min\highlight{\{}k,n-k\highlight{\}}-d/2+1)}.
  \] If $d>2\min(k,n-k)$, then we have $M(q,k,n,d)=1$.
\end{theorem}

The subspace distance \highlight{between} two subspaces with the same pivots can be computed by the rank
distance of the corresponding generator matrices.
\highlightgreen{Using $\tau$ from~(\ref{eqn:tau}), we have:}
\begin{lemma} (\cite[Corollary 3]{MR2801585})
  \label{lemma_same_ef}
  Let $v\in\mathbb{F}_2^n$ and $\highlight{\tau(U),\tau(W)}\in \EF{v}{q}$, then
  $d_S(U,W)=2\cdot d_R\Big(\tau(U),\tau(W)\Big)$.
\end{lemma}

So, in order to construct \highlight{an} $(n,M,2\delta;k)$ code, it suffices to select
a subset of $\EF{v}{q}$ with minimum rank distance $\delta$.
\highlight{Additionally, we can further expand such a code by introducing codewords with different pivot columns as long
as the sets of pivot columns are sufficiently apart.}
Let
$d_H(v,v'):=\left|\left\{1\le i\le n\,:\,v_i\neq v_i'\right\}\right|$
denote the \emph{Hamming distance} for two binary vectors $v,v'\in\F{n}{2}$.

\begin{lemma} (\cite[Lemma 2]{etzion2009error})
  \label{lemma_different_ef}
  Let $v,v'\in\F{n}{2}$, $U\in \EF{v}{q}$, and $W\in \EF{v'}{q}$, then
  $d_S(U,W)\ge d_H(v,v')$.
\end{lemma}

Having Lemma~\ref{lemma_same_ef} and Lemma~\ref{lemma_different_ef} at hand, the
Echelon-Ferrers construction from~\cite{etzion2009error} works as follows:
For two integers $k$ and $\delta$ choose a binary
constant weight code $\mathcal{S}$ of length $n$, weight $k$, and minimum
Hamming distance $2\delta$ as a so-called \emph{skeleton code}. For each
$s\in \mathcal{S}$ construct a code $\mathcal{C}_s\subseteq \EF{s}{q}$
having a minimum rank distance of at least $\delta$. Setting
$\mathcal{C}=\cup_{s\in\mathcal{S}} \mathcal{C}_s$ yields \highlight{an} $(n,M,2\delta;k)$
code\highlight{, where $M=\sum_{s \in S} |C_s|$}.
\highlightgreen{We remark that Lemma~\ref{lemma_different_ef} does not need two
binary vectors $v,v'$ of the same weight, i.e., the very same approach can be used
to \highlight{construct general subspace codes in which the codewords may have different dimensions}.
The only necessary modification is to choose
a general binary code $\mathcal{S}$ of length $n$ and minimum
Hamming distance $d$ as skeleton code. The codes $\mathcal{C}_s$ need to
have a rank distance of at least $d/2$.
}

For a given binary vector $v\in\mathbb{F}_2^n$ and an integer $1\le\delta\le n$ let
$q^{\dim(v,\delta)}$ be the largest cardinality of a linear rank-metric code over
$\EF{v}{q}$ with rank distance at least $\delta$.
\begin{theorem} (\cite[Theorem 1]{etzion2009error})
  \label{thm_mrd_ferrers}
  For a given $i$, $0\le i\le \delta-1$, if $\nu_i$ is the number of dots in the
  Echelon-Ferrers diagram corresponding to $v$, which are not contained in the
  first $i$ rows and not contained in the rightmost $\delta-1-i$ columns, then
  $\min_i\{\nu_i\}$ is an upper bound of $\dim(v,\delta)$.
\end{theorem}

The conjecture that the upper bound of Theorem~\ref{thm_mrd_ferrers} can be obtained
for all parameters is still unrefuted \highlight{and valid in many cases, see~\cite{MR3480069}}.
Several of the currently best known lower bounds
for constant dimension codes are obtained via the Echelon-Ferrers construction.
We remark that for the special binary vector $v=(1,\dots,1,0,\dots,0)$ of length $n$ and
weight $k$, the rank-metric codes of maximum cardinality in $\EF{v}{q}$ are given by
lifted MRD codes, see Theorem~\ref{thm_lifted_MRD}. So, the Echelon-Ferrers construction
uses building blocks \highlight{that are lifted MRD codes with a prescribed structure.}
\highlight{It} is possible to improve the best currently known upper bounds on
$A_q(n,d;k)$ for constant dimension codes that contain \highlight{a} lifted MRD code.

\begin{theorem}(see \cite[\highlight{Theorems} 10 and 11]{etzion2013codes})
\label{thm_MRD_upper_bound}
Let $\mathcal{C}\subseteq \G{n}{k}{q}$, where $n\ge 2k$, with minimum
subspace distance $d$ \highlight{contain a lifted MRD code.}
\begin{itemize}
\item If $d=2(k-1)$ and $k \ge 3$, then $\left|\mathcal{C}\right| \le q^{2(n-k)} + A_q(n-k,2(k-2);k-1)$;
\item if $d=k$, where $k$ is even, then $\left|\mathcal{C}\right| \le q^{(n-k)(k/2+1)} +
      \gaussm{n-k}{k/2}{q}\frac{q^n-q^{n-k}}{q^{k}-q^{k/2}} + A_q(n-k,k;k)$.
\end{itemize}
\end{theorem}

\subsection{\texorpdfstring{Parallelisms and packings of $\G{n}{k}{q}$}{Parallelisms and packings of Gqnk}}
\label{subsec_packing}
Let $X$ be a set. A \emph{packing} $P=\left\{P_1,\dots,P_l\right\}$ of $X$ is a set of subsets $P_i\subseteq X$
such that $P_i\cap P_j=\emptyset$ for all $1\le i<j\le l$, i.e., the subsets $P_i$ are pairwise disjoint.
\highlight{A \emph{point} is an element of $\G{n}{1}{q}$ and a} \emph{spread} is a subset of $\G{n}{k}{q}$
that partitions the corresponding set of points, i.e.,
the elements have a pairwise trivial intersection. Counting the points yields that the size of a spread
is $\frac{\gaussm{n}{1}{q}}{\gaussm{k}{1}{q}}=\frac{q^n-1}{q^k-1}$. A spread is a special constant dimension
code with subspace distance $d=2\cdot k$. Spreads exist if and only if $k$ divides $n$, see \cite{andre1954nicht}.
With this, a \emph{parallelism} in $\G{n}{k}{q}$ is a packing of spreads such that it partitions $\G{n}{k}{q}$.

Parallelisms in $\G{n}{k}{q}$ are known to exist for:
\begin{enumerate}
\item $q=2, k=2$ and $n$ even;
\item $k=2$, all $q$ and $n=2^m$ for $m\ge 2$;
\item $n=4$, $k=2$, and $q\equiv 2\pmod 3$;
\item $q=2, k=3, n=6$,
\end{enumerate}
see e.g.\ \cite{etzion2015galois}.

\section{The coset construction}
\label{sec_coset_construction}

\highlight{Construction~III in \cite{etzion2013codes} gives $A_2(8,4;4)\ge 4797$.
While this specific construction does not involve parameters, the authors conjecture
that the underlying idea can be generalized to further parameters assuming the existence
of a corresponding parallelism. In Theorem~\ref{thm_coset_parallelism} in
Subsection~\ref{subsec_8_4_4} we will show that this is indeed the case. Moreover,
there is a more general underlying construction, introduced as \textit{coset construction}
in this paper, that yields improvements of the
best-known lower bounds for constant dimension codes, see Section~\ref{sec_examples}.}

\highlight{The main idea of the coset construction is to use a collection of
codewords which will be part of a subspace code such that $\tau$ \highlightgreen{form~(\ref{eqn:tau})}, \highlight{i.e., the corresponding \highlightgreen{RRE} form,}
of each element of this collection is of the form }
\[
  \begin{pmatrix} A & \varphi_B(F)\\0 & B\end{pmatrix}.
\]
\highlight{Here, $A$ is the \highlightgreen{RRE} form of a $k'$-dimensional subspace in $\mathbb{F}_q^{n'}$ and $B$
is the \highlightgreen{RRE} form of a $k-k'$-dimensional subspace in $\mathbb{F}_q^{n-n'}$, so that we obtain a \highlightgreen{RRE} form
of a $k$-dimensional subspace $C(A,B,F)$ of $\mathbb{F}_q^{n}$.
\highlightgreen{Note that the integers $k'$ and $n'$ are respectively the same for any codeword in this collection although Lemma~\ref{lemma_use_2_coset_parts} allows to combine multiple such collections.}
$F$ is an arbitrary
$k'\times (n-n'-k+k')$ matrix over $\mathbb{F}_q$, in which $\varphi_B$ inserts zero columns at the
pivot positions of $B$, see Subsection~\ref{subsec_gaussian_elimination} for the precise definition
of $\varphi_B$ and an example. In $C(A,B,F)$,} the vectors have
the shape $(\lambda \cdot A , \lambda \cdot F + \mu \cdot B)$.
So $\lambda \cdot F$ is the offset for the coset of the suffixes, i.e.,
the vector $\lambda \cdot A$ is prefix for every vector in the coset
$\lambda \cdot F + B$\highlightgreen{, explaining the naming of our construction}. \highlight{In order to obtain a constant dimension code with large minimum subspace distance, }
the matrices $A$, $B$, and $F$, \highlight{as well as their combinations, are chosen} from certain \highlightgreen{sets.}
\highlightgreen{Using $\tau$ from~(\ref{eqn:tau}), we have:}
\begin{lemma}
  \label{lemma_coset_construction}
  (Coset construction)
  Let $q$ be a prime power and $n,k,n',k'\in\mathbb{N}$ satisfy
  $1\le k\le n/2$, $1\le k'\le n'$, and $1\le k-k'\le n-n'$. Let
  further $\mathcal{A}=\dot\cup_{1\le i\le l} \mathcal{A}_i$,
  $\mathcal{B}=\dot\cup_{1\le i\le l} \mathcal{B}_i$, where
  $\emptyset\neq \mathcal{A}_i\subseteq\G{n'}{k'}{q}$ and $\emptyset\neq\mathcal{B}_i\subseteq\G{n-n'}{k-k'}{q}$
  for all $1\le i\le l$, and $\overline{F}\subseteq \mathbb{F}_q^{k'\times (n-n'-k+k')}$.
  With this, we have that $\mathcal{C}\left(\left(\mathcal{A}_i\right)_i,\left(\mathcal{B}_i\right)_i,\overline{F}\right):=$
  \[
    \left\{\tau^{-1}\begin{pmatrix}A&\varphi_B(F)\\0&B\end{pmatrix}\,:\,\tau^{-1}(A)\in\mathcal{A}_i,
    \tau^{-1}(B)\in\mathcal{B}_i,1\le i\le l, F\in\overline{F}\right\}
  \]
  is a subset of $\G{n}{k}{q}$, i.e., a constant dimension code where the codewords have
  dimension $k$.
\end{lemma}
\begin{proof}
  For an arbitrary but fixed index $1\le i\le l$ let $A$, $B$ be matrices with
  $\tau^{-1}(A)\in\mathcal{A}_i$ and $\tau^{-1}(B)\in\mathcal{B}_i$. We can easily
  check that $A\in\mathbb{F}_q^{k'\times n'}$ is a full-rank matrix in \highlightgreen{RRE} form.
  Similarly, $B\in\mathbb{F}_q^{(k-k')\times (n-n')}$ is a full-rank matrix in \highlightgreen{RRE} form.
  For each matrix $F\in\overline{F}$ we have $F\in\mathbb{F}_q^{k'\times(n-n'-k+\highlight{k')}}$, so that
  $\varphi_B(F)\in \mathbb{F}_q^{k'\times(n-n')}$. The dimensions fit so that
  \[
    M:=\begin{pmatrix}A&\varphi_B(F)\\0&B\end{pmatrix}\in \mathbb{F}_q^{k\times n}.
  \]
  Moreover $\varphi_B(F)$ has zero columns at the positions of the pivot columns of $B$. Since
  $A$ has $k'$ and $B$ has $k-k'$ pivot columns, $M$ has exactly $k$ pivot columns and full rank.
  Thus, $\tau^{-1}(M)\in\G{n}{k}{q}$.
\end{proof}

The number $l$ of disjoint subsets for $\mathcal{A}$ and $\mathcal{B}$ is called the
\emph{length} of the specific coset construction.
We remark that we have excluded the ranges for the parameters $k',n'$ where the construction
would be degenerated in the sense that either $A$ or $B$ have to be empty matrices. Nevertheless,
the \highlight{degenerate} case $k'=k$ has a nice interpretation. Here $B$ is an empty matrix and $A$ is a
$k\times n'$ matrix. If additionally $n'=k$ then $A$ is an identity matrix and we are in the case
of lifted MRD codes.
\highlightgreen{Using $\tau$ from~(\ref{eqn:tau}), we have:}
\begin{lemma}
  \label{lemma_dist_coset_construction}
  Let $q,n,k,n',k'$ be parameters satisfying the conditions from Lemma~\ref{lemma_coset_construction},
  $A,A'\in\mathbb{F}_q^{k'\times n'}$ and $B,B'\in \mathbb{F}_q^{(k-k')\times (n-n')}$ be full-rank
  matrices in \highlightgreen{RRE} form. Let further $d$ be a positive integer and $F,F'\in \mathbb{F}_q^{k'\times(n-n'-k+k')}$.
  If
  \begin{equation}
    \label{ie_combined_distance_condition}
    d_S(\tau^{-1}(A),\tau^{-1}(A'))+d_S(\tau^{-1}(B),\tau^{-1}(B')) \ge d
  \end{equation}
  or $d_R(F,F') \ge d/2$ then
  \begin{align*}
  d_S\left(\tau^{-1}
  \begin{pmatrix}A&\varphi_B(F)\\0&B\end{pmatrix}
  ,\tau^{-1}
  \begin{pmatrix}A'&\varphi_{B'}(F')\\0&B'\end{pmatrix}
  \right)\ge d.
\end{align*}
\end{lemma}

\highlight{The proof is rather technical and can be found in the appendix.}

We remark that condition~(\ref{ie_combined_distance_condition}) of Lemma~\ref{lemma_dist_coset_construction}
is trivially satisfied for the special case of distance $d=4$\highlight{, if $A \ne A'$ and $B \ne B'$}.

Next we demonstrate that the coset construction from Lemma~\ref{lemma_coset_construction}
can in general not be obtained by an application of the Echelon-Ferrers construction.
\highlight{(For a more explicit example, see Theorem~\ref{thm_improve_10_6_4} in Subsection~\ref{subsec_i_10_6_4}.)}
It is easy to construct a family of examples with subspace distance $d$ but whose pivot vectors
have Hamming distance $2$, so that they cannot be used in the Echelon-Ferrers
construction. To this end, let $q$ be an arbitrary prime power, $d$ an even integer $\ge 2$,
and $n,k,n',k'\in\mathbb{N}$ such that $\frac{d}{4} \le k',n'-k',k-k',n-n'-k+k'$.
For the sake of this example we use:
\begin{align*}
A_1 &:= \begin{pmatrix} I_{k'-1} & 0 & M & 0 \\ 0 & 1 & 0 & 0 \end{pmatrix} & A_2 &:= \begin{pmatrix} I_{k'-1} & 0 & M+N & 0 \\ 0 & 0 & 0 & 1 \end{pmatrix} \\
B_1 &:= \begin{pmatrix} I_{k-k'} & M' \end{pmatrix} & B_2 &:= \begin{pmatrix} I_{k-k'} & M' + N' \end{pmatrix}
\end{align*}
with arbitrary matrices $M , N \in \mathbb{F}_q^{(k'-1) \times (n'-k'-1)}$ of full rank,

$M',N' \in \mathbb{F}_q^{(k-k')
\times (n-n'-k+k')}$, where $I_\star$ denotes the identity matrix.
Then, for arbitrary $F_1,F_2 \in \mathbb{F}_q^{k' \times (n-n'-k+k')}$:
\[
d_H\left(
p\left(\begin{pmatrix}A_1&F_1\\0&B_1\end{pmatrix}\right),
p\left(\begin{pmatrix}A_2&F_2\\0&B_2\end{pmatrix}\right)
\right) = 2
\]
but
\begin{align*}
&d_S\left(
\begin{pmatrix}A_1&F_1\\0&B_1\end{pmatrix},
\begin{pmatrix}A_2&F_2\\0&B_2\end{pmatrix}
\right)\\
&\ge 2\left(
\rk\begin{pmatrix}A_1\\A_2\end{pmatrix}
+
\rk\begin{pmatrix}B_1\\B_2\end{pmatrix}
-k\right)
\\
&= 2((k'+1+\rk(N))+(k-k'+\rk(N'))-k)
\\
&= \highlight{2(
\min\{k' ,n'-k'\}+
\min\{k-k',n-n'-k+k'\}
)}
\\
&
\ge d.
\end{align*}

\subsection{A multilevel coset construction}
\label{subsec_multilevel}

In this subsection we want to use the coset construction
in \highlight{combination} with other \highlight{constructions}. At first we show
that it is compatible with the Echelon-Ferrers construction.
\highlightgreen{Using $\tau$ from~(\ref{eqn:tau}), we have:}
\begin{lemma}
  \label{lemma_multilevel_coset_hamming}
  \highlight{
  For a prime power $q$ and $n,k,n',k',\tilde{k}\in\mathbb{N}$ satisfying
  $1\le k\le n/2$, $1\le k'\le n'$, $1\le k-k'\le n-n'$, and $0 \le \tilde{k} \le n$,
  let} $U \in \G{n'}{k'}{q}$, $V \in \G{n-n'}{k-k'}{q}$, $F \in \mathbb{F}_q^{k' \times \highlight{(n-n'-k+k')}}$,
  and $X \in \G{n}{\tilde{k}}{q}$. Let $s$ be the sum of the first $n'$ entries in the pivot vector $p(X)$ of $X$,
  i.e., $s:= \sum_{i=1}^{n'}p(X)_i$. If $d \le \left|s-k'\right|+\left|\tilde{k}-s-k+k'\right|$
  then $d_S\left(X,W\right)\ge d$, where
  \[
    W=\tau^{-1}\left(\begin{pmatrix}\tau(U)&\varphi_{\tau(V)}(F)\\0&\tau(V)\end{pmatrix}\right).
  \]
\end{lemma}
\begin{proof}
  Let $x:=p(X)$ and $w:=p(W)$ \highlight{be} the pivot vectors of $X$ and $W$, respectively.
  From the construction we know $\sum_{i=1}^{n'} x_i=s$, $\sum_{i=1}^{n'} w_i=k'$,
  $\sum_{i=n'+1}^{n} x_i=\tilde{k}-s$, and $\sum_{i=n'+1}^{n} w_i=k-k'$, so that
  \[
    d_H(x,w)\ge \left|s-k'\right|+\left|(\tilde{k}-s)-(k-k')\right|\ge d.
  \]
  Applying Lemma~\ref{lemma_different_ef} yields the stated lower bound on the subspace distance.
\end{proof}

For the special case $\tilde{k}=k$, i.e., the constant dimension case, we have
$\left|s-k'\right|+\left|\tilde{k}-s-k+k'\right|=2\cdot\left|s-k'\right|$.
There is also an easy-to-check sufficient criterion \highlight{to determine} whether the union of two
codes constructed by the coset construction have a subspace distance of at least $d$.

\begin{lemma}
  \label{lemma_use_2_coset_parts}
  Let $\mathcal{C}_i$ be codes having subspace distance at least $d$ and that are obtained from
  the coset construction with suitable parameters $n$, $k_i$, $n'_i$, and $k'_i$ for $i=1,2$,
  where we assume \highlight{$n_1'\le n_2'$}. Let $f(m)=|m-k_1'|+|m-\gamma|$
  and

\highlight{  \[
    K=
    \begin{cases}
      f(\overline{\beta}) & \text{if } \overline{\beta} \le \lambda \\
      f(\lambda) & \text{if } \underline{\beta} \le \lambda < \overline{\beta} \\
      f(\underline{\beta}) & \text{else}
    \end{cases}\text{,}
  \]}

  where $\underline{\beta}=\max\{\highlight{k_2'-n_2'+n_1'},0\}$, $\overline{\beta}=\min\{\highlight{n_1'},k_2'\}$,
  $\gamma=k_1'+k_2-k_1$, \highlight{and $\lambda=\max\{\gamma,k_1'\}$}.
  If \highlight{$d \le K$},
  then $D_S\left(\mathcal{C}_1 \cup \mathcal{C}_2\right) \ge d$.
\end{lemma}
\begin{proof}
  At first we observe that we have $d_H(u,v)\ge |a-b|$ for $u,v\in\mathbb{F}_2^n$
  with $\Vert u\Vert_1=a$ and $\Vert v\Vert_1=b$.

  We set $x:=p\left(W_1\right)$ and $y:=p\left(W_2\right)$, where \highlightgreen{the}
  $W_i$ are matrices corresponding to an arbitrary but fixed codeword from
  \highlight{$\mathcal{C}_i$, see} the formulation of Lemma~\ref{lemma_multilevel_coset_hamming}.

  Let $x^1$ consist of the first \highlight{$n_1'$} entries of $x$,
  $y^1$ consist of the first \highlight{$n_1'$} entries of $y$, $x^2$ consist of the last \highlight{$n-n_1'$} entries of $x$,
  and $y^2$ consist of the last \highlight{$n-n_1'$} entries of $y$. \highlight{For} $m:=\Vert y^1\Vert_1$, where
  \[
    \underline{\beta}=\max\{\highlight{k_2'-n_2'+n_1'},0\}\le m\le \min\{\highlight{n_1'},k_2'\}=\overline{\beta},
  \]
  we have $d_H\left(x^1,y^1\right)\ge |m-k_1'|$
  and $d_H\left(x^2,y^2\right)\ge |m-\gamma|$.
  Thus $f(m)\le d_H(x,y)$ is minimized for \highlight{$K$}. Applying
  Lemma~\ref{lemma_different_ef} yields the stated lower bound on the subspace distance.
\end{proof}

\highlight{
Considering the exemplary parameters $n=6$, $k_1=k_2=3$, $n_1'=n_2'=2$, $k_1'=1$ and $k_2'=2$ for the codes
$C_1=\left\{\left(
\begin{smallmatrix}
1&0&0&0&0&0\\
0&0&1&0&0&0\\
0&0&0&0&0&1\\
\end{smallmatrix}
\right)\right\}$
and
$C_2=\left\{\left(
\begin{smallmatrix}
1&0&0&0&0&0\\
0&1&0&0&0&0\\
0&0&1&0&0&0\\
\end{smallmatrix}
\right)\right\}$
Lemma~\ref{lemma_use_2_coset_parts} uses $\underline{\beta}=\overline{\beta}=2$, $\gamma=\lambda=1$,
$f(m)=2|m-1|$ and $K=f(\underline{\beta})=2$. This lower bounds any two codewords from two coset constructed
parts having these parameters, whereas the Hamming distance of the depicted pivot vectors is $2$.
}

We remark that Lemma~\ref{lemma_use_2_coset_parts} is best possible in the
sense that the estimations on the Hamming distance of two binary
vectors with known weights and weights of two suffixes, of possibly different
lengths, is tight. Performing similar analyses on generalized structures like
\[
  \begin{pmatrix} A & \varphi_{B}(F) & \varphi_{C}(G)\\0 & B & \varphi_{C}(H)\\ 0 & 0 & C\end{pmatrix}
\]
may have the potential to yield stronger bounds.

\section{Optimal choices for the parameters of the coset construction}
\label{sec_optimal_parameters}

\subsection{\highlightgreen{General reasoning}}

\highlight{Like the Echelon-Ferrers construction, the} coset construction from the previous section is far from being
explicit, i.e., there are several degrees of freedom. \highlight{In this section
we give several lower and upper bounds for the sizes of the codes obtained from the
coset construction, which allow to minimize the range of choices of the parameters
that can lead to improvements of the best-known bounds.}

The cardinality
of a subspace code obtained from the coset construction with length $l$ is given by
\begin{equation}
  \label{eq_cardinality_coset_construction}
  \left|\mathcal{C}\left(\left(\mathcal{A}_i\right)_i,\left(\mathcal{B}_i\right)_i,\overline{F}\right)\right|
  =\left|\overline{F}\right|\cdot \overset{\Lambda:=}{\overbrace{\sum_{i=1}^{l} \left|\mathcal{A}_i\right| \cdot \left|\mathcal{B}_i\right|}}.
\end{equation}
Given $q$, $n$, and the desired even subspace distance $d$, the aim is to maximize
(\ref{eq_cardinality_coset_construction}) under the restrictions of Lemma~\ref{lemma_dist_coset_construction}.
Obviously, this term is maximal if both $\left|\overline{F}\right|$ and the sum are maximal. Thus, we may
choose an MRD code, with appropriate parameters, for $\overline{F}$, so that
\[
  \left|\overline{F}\right|=\left\lceil q^{\max\{k',n-n'-k+k'\}\cdot(\min\{k',n-n'-k+k'\}-d/2+1)}\right\rceil
\]
\highlight{is optimal by} Theorem~\ref{thm_MRD_size}.

The sets $\mathcal{A}_i$ and $\mathcal{B}_i$ need to have additional structure.

\begin{lemma}
  \label{lemma_cdc_part}
  For a code obtained from the construction of Lemma~\ref{lemma_coset_construction}
  with $d:=D_S\Big(\mathcal{C}\left(\left(\mathcal{A}_i\right)_i,\left(\mathcal{B}_i\right)_i,\overline{F}\right)\Big)$,
  length $l$, and parameters $q,n,k,n',k'$ we have $D_S\left(\mathcal{A}_i\right)\ge d$ and
  $D_S\left(\mathcal{B}_i\right)\ge d$ for all $1\le i\le l$.
\end{lemma}
\begin{proof}
  If $U\neq U'\in \mathcal{A}_i$, then there exists $V\in\mathcal{B}_i$ such that
  Condition~(\ref{ie_combined_distance_condition}) yields
  $d\le d_S(U,U')+d_S(V,V)=d_S(U,U')$. A similar conclusion can be drawn for
  the elements in $\mathcal{B}_i$.
\end{proof}

From this we can conclude an upper bound on $\Lambda$.
\begin{corollary}
  \label{cor_upper_bound_Lambda}
  Using the notation from Lemma~\ref{lemma_coset_construction} and
  Equation~(\ref{eq_cardinality_coset_construction}) we have
  \[
    \Lambda\le \min\left\{
    \gaussm{n'}{k'}{q}\cdot A_q(n-n',d;k-k'),
    \gaussm{n-n'}{k-k'}{q}\cdot A_q(n',d;k')
    \right\}.
  \]
\end{corollary}
\begin{proof}
  Due to Lemma~\ref{lemma_cdc_part} we have $\left|\mathcal{A}_i\right|\le A_q(n',d;k')$, so that
  \begin{eqnarray*}
    \sum_{i=1}^{l} \left|\mathcal{A}_i\right| \cdot \left|\mathcal{B}_i\right|
    \le A_q(n',d;k')\cdot \sum_{i=1}^{l} \left|\mathcal{B}_i\right|
    \le A_q(n',d;k')\cdot \gaussm{n-n'}{k-k'}{q}.
  \end{eqnarray*}
  Interchanging the roles of the $\mathcal{A}_i$ and $\mathcal{B}_i$ yields
  the other stated upper bound.
\end{proof}

\begin{corollary}
  The upper bound of Corollary~\ref{cor_upper_bound_Lambda} can be attained
  if $d\le 4$ and both $\G{n'}{k'}{q}$ and $\G{n-n'}{k-k'}{q}$
  admit parallelisms\highlightgreen{, e.g., the corresponding parameters are in the list in Subsection~\ref{subsec_packing}}.
\end{corollary}

The dependency between the cardinalities of the $\mathcal{A}_i$ and
$\mathcal{B}_i$ in optimal solutions of (\ref{eq_cardinality_coset_construction})
is already decoupled to some \highlight{extent}, but we can even do more.

\begin{lemma}
  \label{lemma_min_dist_union}
  For a code obtained from the construction of Lemma~\ref{lemma_coset_construction}
  with $d:=D_S\Big(\mathcal{C}\left(\left(\mathcal{A}_i\right)_i,\left(\mathcal{B}_i\right)_i,\overline{F}\right)\Big)$,
  length $l$, and parameters $q,n,k,n',k'$, there exists an integer $d'$ such that
  $D_S(\mathcal{A})\ge d'$ and $D_S(\mathcal{B})\ge d-d'$, where $\mathcal{A}=\cup_i\mathcal{A}_i$
  and $\mathcal{B}=\cup\mathcal{B}_i$.
\end{lemma}
\begin{proof}
  Let $U,U'\in\mathcal{A}$ with $d_S(U,U')=D_S(\mathcal{A})=:d'$ and
  $V,V'\in\mathcal{B}$ with $d_S(V,V')=D_S(\mathcal{B})=:d''$.
  \highlight{W.l.o.g. we can assume that $\overline{F}$ contains the zero matrix,
  since the rank distance is invariant with respect to translations.}
  Choosing $F=F'=\mathbf{0}$ we can conclude $d''\ge d-d'$ from
  Inequality~(\ref{ie_combined_distance_condition}).
\end{proof}

\highlight{
In later applications we will commonly assume $2\le d'\le d-2$, since the other values lead
to trivial cases where either $|\mathcal{A}|=1$ or $|\mathcal{B}|=1$.
}

\begin{lemma}
  \label{lemma_permutation_sizes}
  For a code obtained from the construction of Lemma~\ref{lemma_coset_construction}
  with $d:=D_S\Big(\mathcal{C}\left(\left(\mathcal{A}_i\right)_i,\left(\mathcal{B}_i\right)_i,\overline{F}\right)\Big)$,
  length $l$, and parameters $q,n,k,n',k'$, then for each permutation $\highlight{\sigma}:\{1,\dots,l\}\rightarrow\{1,\dots,l\}$
  we have $D_S\Big(\mathcal{C}\left(\left(\mathcal{A}_i\right)_i,\left(\mathcal{B}_{\highlight{\sigma}(i)}\right)_i,\overline{F}\right)\Big)=d$.
\end{lemma}
\begin{proof}
  Apply Lemma~\ref{lemma_dist_coset_construction}.
\end{proof}

The question which permutation $\highlight{\sigma}$ of Lemma~\ref{lemma_permutation_sizes} maximizes
the crucial parameter $\Lambda$ can be answered easily.

\begin{lemma}
  \label{lemma_best_permutation}
  Let $a_1\ge\dots\ge a_l$ and $b_1\ge \dots\ge b_l$ positive integers. For each
  permutation $\highlight{\sigma}:\{1,\dots,l\}\rightarrow \{1,\dots,l\}$, we have
  \[
    \sum_{i=1}^{l} a_i\cdot b_i\ge \sum_{i=1}^{l} a_i\cdot b_{\highlight{\sigma}(i)}.
  \]
\end{lemma}
\begin{proof}
  For integers $a>a'$ and $b<b'$ we have
  \[
    (ab+a'b')-(ab'+a'b)=(a-a')\cdot(b-b')<0.
  \]
\end{proof}

Having these ingredients at hand we can generalize and improve the upper bound
from Corollary~\ref{cor_upper_bound_Lambda} \highlight{using} the analytical solution
of another optimization problem.
\begin{lemma}
  \label{lemma_optimal_ilp_solution}
  Let $\alpha$, $\beta$, $\overline{\alpha}$, $\overline{\beta}$, and $l$ be positive integers
  with $\alpha,\beta\ge l$.
  An optimal solution of the non-linear integer programming problem
  \begin{eqnarray*}
    \max \sum_{i=1}^{l} a_i\cdot b_i&&\\
    \sum_{i=1}^{l} a_i\le \alpha
    &&\sum_{i=1}^{l} b_i\le \beta\\
    1\le a_i\le \overline{\alpha}\quad\forall 1\le i\le l
    && 1\le b_i\le \overline{\beta}\quad\forall 1\le i\le l\\
    a_i,b_i\in\highlight{\mathbb{N}_{> 0}}\quad\forall 1\le i\le l
  \end{eqnarray*}
  is given by
  \begin{enumerate}
    \item[(1)] $a_i^\star=\overline{\alpha}$, $b_i^\star=\overline{\beta}$ for all
               $1\le i\le l$ if $\overline{\alpha}\cdot l\le \alpha$ and
               $\overline{\beta}\cdot l\le \beta$;
    \item[(2)] $a_i^\star=\overline{\alpha}$, $b_i^\star=1+\min\{\overline{\beta}-1,\max\{0,\beta-l-(i-1)\cdot\left(\overline{\beta}-1\right)\}\highlight{\}}$
               for all $1\le i\le l$ if $\overline{\alpha}\cdot l\le \alpha$ and
               $\overline{\beta}\cdot l> \beta$;
    \item[(3)] $a_i^\star=1+\min\{\overline{\alpha}-1,\max\{0,\alpha-l-(i-1)\cdot\left(\overline{\alpha}-1\right)\}\highlight{\}}$,
               $b_i^\star=\overline{\beta}$ for all $1\le i\le l$ if $\overline{\alpha}\cdot l> \alpha$ and
               $\overline{\beta}\cdot l\le \beta$;
    \item[(4)] $a_i^\star=1+\min\{\overline{\alpha}-1,\max\{0,\alpha-l-(i-1)\cdot\left(\overline{\alpha}-1\right)\}\highlight{\}}$,
               $b_i^\star=1+\min\{\overline{\beta}-1,\max\{0,\beta-l-(i-1)\cdot\left(\overline{\beta}-1\right)\}\highlight{\}}$
               for all $1\le i\le l$ if $\overline{\alpha}\cdot l> \alpha$ and
               $\overline{\beta}\cdot l> \beta$.
  \end{enumerate}
\end{lemma}
\begin{proof}
  W.l.o.g.\ we can additionally assume $a_1\ge\dots\ge a_l$ and $b_1\ge\dots\ge b_l$ without
  decreasing the maximal target value of the optimization problem. Let us allow $a_i,b_i\in\mathbb{R}$
  for a moment, i.e., we consider the standard relaxation, and denote a corresponding optimal solution
  by $\tilde{a}_i$, $\tilde{b}_i\highlight{\in\mathbb{R}_{\ge 1}}$.

  For \highlight{non-negative} real numbers $a'\ge a''$ and $b'\ge b''$ we have
  \[
    \left(a'b'+a''b''\right)-2\cdot\frac{a'+a''}{2}\cdot\frac{b'+b''}{2}
    =\frac{(a'-a'')\cdot(b'-b'')}{2}\ge 0,
  \]
  so that we can assume $\tilde{a}_i=\tilde{a}_j=:\tilde{a}$ and $\tilde{b}_i=\tilde{b}_j=:\tilde{b}$,
  for all $1\le i,j\le l$, w.l.o.g.

  Either we have $l\tilde{a}=\alpha$ or $\tilde{a}=\overline{\alpha}$, since otherwise we could slightly
  increase $\tilde{a}$ and improve the target value. The same reasoning applies to $\tilde{b}$.

  If $\tilde{a}=\overline{\alpha}$ and $\tilde{b}=\overline{\beta}$, then we are in case (1). Next we consider
  the case where $\tilde{a}=\overline{\alpha}$ and $\tilde{b}<\overline{\beta}$ so that $\tilde{b}=\beta/l$.
  Since $\sum_{i=1}^l b_i\overline{\alpha}= \overline{\alpha}\cdot \sum_{i=1}^l b_i$ it suffices to determine
  integers $1\le b_i^\star\le\overline{\beta}$ with $\sum_{i=1}^{l} b_i^\star=\beta$. This is done in the
  formula of case (2). The underlying idea is the following: Start with $b_i^\star=1$ for all $1\le i\le l$;
  observe $\beta\ge l$. Then fill up the $b_i^\star$ with increasing indices up to $\overline{\beta}$ as
  long as the sum does not violate $\beta$. \highlight{Observe that every (integer) vector $(b_i)$ with
  $\sum_{i=1}^l b_i=\beta$ gives the same target value.} Case (3) describes the symmetric situation. It remains to assume
  $\overline{\alpha}\cdot l> \alpha$ and $\overline{\beta}\cdot l> \beta$. Let $\hat{a}_i$, $\hat{b}_i$ be
  an optimal solution of our initial optimization problem where we assume $\hat{a}_1\ge \dots\ge \hat{a}_l$ and
  $\hat{b}_1\ge \dots\ge \hat{b}_l$. Let further $f$ be the smallest index such that $\hat{a}_f<\overline{\alpha}$
  and $r$ be the largest index such that $\hat{a}_r>1$. If either $r$ does not exist or $f=r$ \highlight{($f$ exists due
  to $\bar{\alpha}\cdot l>\alpha$)}, then the
  solution $\hat{a}_i$ has the shape described in case~(4). But, for $f<r$ we could improve the target value by
  \[
    \left(\hat{a}_f+1\right)\cdot \hat{b}_f+\left(\hat{a}_r-1\right)\cdot \hat{b}_r
    -\hat{a}_f\cdot\hat{b}_f -\hat{a}_r\cdot\hat{b}_r=\hat{b}_f-\hat{b}_r\ge 0,
  \]
  so that \highlight{such a case could never produce an optimal value and so our solution must have the
  shape described in case~(4)}. The same reasoning applies for the $\hat{b}_i$.
\end{proof}
\begin{lemma}
  \label{lemma_upper_bound_Lambda}
  Using the notation from Lemma~\ref{lemma_coset_construction} and
  Equation~(\ref{eq_cardinality_coset_construction}) we have
  \begin{eqnarray*}
    \Lambda&\le&\max_{d'\in 2\mathbb{Z}\,:\,0<d'<d}
    \,\,\,\,\max_{1\le l\le \min\left\{A_q(n',d';k'),A_q(n-n',d-d';k-k')\right\}}\,\,\,\,
    \sum_{i=1}^l a_i\cdot b_i,
  \end{eqnarray*}
  where the $a_i$, $b_i$ are given by Lemma~\ref{lemma_optimal_ilp_solution} for
  \begin{eqnarray*}
    \alpha&=& A_q(n',d';k'),\\
    \beta&=& A_q(n-n',d-d';k-k'),\\
    \overline{\alpha}&=&A_q(n',d;k'),\\
    \overline{\beta}&=&A_q(n-n',d;k-k').
  \end{eqnarray*}
\end{lemma}
\begin{proof}
  From Lemma~\ref{lemma_min_dist_union} we conclude $\left|\mathcal{A}\right|\le A_q(n',d';k')$
  and $\left|\mathcal{B}\right|\le A_q(n-n',d-d';k-k')$. The possible values for the length $l$
  are part of the stated optimization formulation. For each index $1\le i\le l$ we have
  $\left|\mathcal{A}_i\right|\le A_q(n',d;k')$ and $\left|\mathcal{B}_i\right|\le A_q(n-n',d;k-k')$
  due to Lemma~\ref{lemma_cdc_part}.
  It remains to check that we can apply Lemma~\ref{lemma_optimal_ilp_solution}.
\end{proof}

Fixing the parameter $d'$ from Lemma~\ref{lemma_min_dist_union} one can state a lower bound on the
maximal value of $\Lambda$ \highlight{in terms of the sizes of lifted MRD codes (cf.\ Theorem~\ref{thm_lifted_MRD})}.

\begin{lemma}
  \label{lemma_primal bound_Lambda}
  Let $d'\in 2\mathbb{Z}$ with $2\le d'\le d-2$, then we have
  \[
    \Lambda \ge M(q,k',n',d) \cdot M(q,k-k',n-n',d) \cdot l
  \]
  with
  \[
    l=\min\left\{\frac{M(q,k',n',d')}{M(q,k',n',d)},\frac{M(q,k-k',n-n',d-d')}{M(q,k-k',n-n',d)}\right\}.
  \]
  for, with respect to Lemma~\ref{lemma_coset_construction}, feasible parameters $q,n,k,n',k',d$.
\end{lemma}
\begin{proof}
  Similar to the proof of \cite[Lemma 5]{etzion2013codes}, we consider $\mathcal{A}$ as a \highlight{linear} MRD code with parameters
  $k'\times n'$ with distance $d'$ and $\mathcal{B}$ as a \highlight{linear} MRD \highlight{code} with parameters
  $(k-k')\times (n-n')$ with distance $d-d'$. Let $\mathcal{S}_{\mathcal{A}}$ be a \highlight{linear} MRD code with parameters
  $k'\times n'$ with distance $d>d'$ and $\mathcal{S}_{\mathcal{B}}$ be a \highlight{linear} MRD code with parameters
  $(k-k')\times (n-n')$ with distance $d>d-d'$. We choose the $\mathcal{A}_i$ as the cosets of $\mathcal{S}_{\mathcal{A}}$ in
  $\mathcal{A}$ and $\mathcal{B}_i$ as the cosets of $\mathcal{S}_{\mathcal{B}}$ in
  $\mathcal{B}$. For $\mathcal{S}_{\mathcal{A}}$ there are exactly $\frac{M(q,k',n',d')}{M(q,k',n',d)}$ cosets and for
  $\mathcal{S}_{\mathcal{B}}$ there are exactly $\frac{M(q,k-k',n-n',d-d')}{M(q,k-k',n-n',d)}$ cosets.
  Since $d_R(A+C,B+C)=d_R(A,B)$ for all suitable matrices $A,B,C\in\mathbb{F}_q^{s\times t}$, we have
  $D_S(\mathcal{A}_i), D_S(\mathcal{B}_i)\ge d$ for all $1\le i\le l$.
\end{proof}
Combining a lifted MRD code with a code constructed from Lemma~\ref{lemma_primal bound_Lambda} yields a
$(9,1032,6;4)_2$ code, which improves \highlight{on} the previously \highlight{best-known} codes, see  Subsection~\ref{subsec_i_9_6_4}.

We can formulate the following greedy-type algorithm to construct sequences $\mathcal{A}_i$ and $\mathcal{B}_i$
that yield a {\lq\lq}reasonable{\rq\rq} lower bound on $\Lambda$.

\addtocounter{algorithm}{\value{theorem}}
\stepcounter{theorem}

\begin{algorithmic}
\captionof{algorithm}{$\,$}\label{algorithm_greedy}
\State \highlightgreen{$\mathcal{R}_{\mathcal{A}}\gets \G {n'}{k'}{q}$}
\State $i\gets 0$
\While {$\mathcal{R}_{\mathcal{A}} \neq \emptyset$}
\State $i\gets i+1$
\State
\begin{varwidth}[t]{\linewidth}
select constant dimension code $\mathcal{A}_i$ of maximum\par\hskip\algorithmicindent cardinality in $\mathcal{R}_{\mathcal{A}}$ with $D_S(\mathcal{A}_i) \ge d$
\end{varwidth}
\State $\mathcal{R}_{\mathcal{A}} \gets \mathcal{R}_{\mathcal{A}} \setminus \{V \mid D_S(\mathcal{A}_i\cup \{V\}) \le d'-1\}$
\EndWhile
\State $l_{\mathcal{A}}\gets i$
\State $\mathcal{R}_{\mathcal{B}}\gets \G {n-n'}{k-k'}{q}$
\State $i\gets 0$
\While {$\mathcal{R}_{\mathcal{B}} \neq \emptyset$}
\State $i\gets i+1$
\State
\begin{varwidth}[t]{\linewidth}
select constant dimension code $\mathcal{B}_i$ of maximum\par\hskip\algorithmicindent cardinality in $\mathcal{R}_{\mathcal{B}}$ with $D_S(\mathcal{B}_i) \ge d$
\end{varwidth}
\State $\mathcal{R}_{\mathcal{B}} \gets \mathcal{R}_{\mathcal{B}} \setminus \{V \mid D_S(\mathcal{B}_i\cup \{V\}) \le d-d'-1\}$
\EndWhile
\State $l_{\mathcal{B}}\gets i$
\State $l\gets \min\left\{l_{\mathcal{A}},l_{\mathcal{B}}\right\}$
\end{algorithmic}

Unfortunately, this algorithm is not capable of determining the optimal $\Lambda$ in general.
If we use
\begin{eqnarray*}
  E:= \{\text{all constant dimension codes in } \G{\tilde{n}}{\tilde{k}}{q} \text{ with subspace distance } d \}
\end{eqnarray*}
as ground set
and $I:=\{\text{disjoint subsets of } E \}$ as independent sets, then this \highlight{does not form a} \emph{matroid} and hence \highlight{a} greedy \highlight{algorithm}
will not yield an optimal solution in general, see e.g.\ \cite{edmonds1971matroids}. To be more precise, the independent set
exchange property fails: Use for example
$U \ne V \in \G{\tilde{n}}{\tilde{k}}{q}$ with $d_S(U,V)\ge d$, $A:=\{\{U\}, \{V\}\} \in I$ and $B:=\{\{U,V\}\} \in I$. Although
$A$ is larger than $B$ we cannot add an element of $A$ to $B$ without losing the independence.

\subsection{Decomposing constant dimension codes}
\label{subsec_decomposition_cdc}

Due to Lemma~\ref{lemma_min_dist_union} we can construct the necessary parts of the coset construction
of Lemma~\ref{lemma_coset_construction} starting from constant dimension codes $\mathcal{A}$ and
$\mathcal{B}$ with $D_S(\mathcal{A})\ge d'$ and $D_S(\mathcal{B})\ge d-d'$. The aim is to partition
the codewords of $\mathcal{A}$ into subcodes $\mathcal{A}_i$ for $1\le i\le l_{\mathcal{A}}$ in such a
way that $D_S\left(\mathcal{A}_i\right)\ge d$. Simultaneously, we aim to partition the codewords
of $\mathcal{B}$ into subcodes $\mathcal{B}_i$ for $1\le i\le l_{\mathcal{B}}$ in such a
way that $D_S\left(\mathcal{B}_i\right)\ge d$. Setting the length $l$ of the coset construction to
$l:=\min\left\{l_{\mathcal{A}},l_{\mathcal{B}}\right\}$, we observe that trying to maximize
the cardinalities $\left|\mathcal{A}_i\right|$ or $\left|\mathcal{B}_i\right|$ for $i>l$ has no benefit,
so that we may simply complete a given packing by singletons. Or, in other words, we directly start from
packings within $\mathcal{A}$ and $\mathcal{B}$.

However, the design of suitable $\mathcal{A}_i$ is not that obvious since the $\Lambda$-part of the
target function~(\ref{eq_cardinality_coset_construction}) comprises a non-linear integer optimization problem.
Ignoring almost all of the geometric restrictions from $\PG{n}{q}$, we are able to exactly solve the mentioned
optimization problem in Lemma~\ref{lemma_optimal_ilp_solution}. In general this gives us an upper bound only.
To obtain tighter bounds one has to go a bit more into the details. In Lemma~\ref{lemma_upper_bound_Lambda}
we have only used the implication $\left|\mathcal{A}_i\right|\le A_q(n',d;k')$ from
$D_S\left(\mathcal{A}_i\right)\ge d$, which is valid for all
$\cup_{i=1}^{l}\mathcal{A}_i\subseteq\mathcal{A}\subseteq\G{n'}{k'}{q}$. For a
given $\mathcal{A}$ we may be able to determine tighter bounds on the cardinalities of the $\mathcal{A}_i$s.
Since the only change in the setting is the exclusion of the possible codewords in $\G{n'}{k'}{q}
\backslash\mathcal{A}$ this subproblem can be formulated as an independent set problem and be solved using
several algorithmic approaches, see e.g.\ \cite{kohnert2008construction}. We will present an explicit example
of this technique in Subsection~\ref{subsec_i_10_6_4}.

Having candidates for the $\mathcal{A}_i$ at hand\highlight{, it still} remains to select a subset of the candidates that are pairwise
disjoint. This subproblem can also be formulated as a (restricted) independent set problem of a, possibly large, graph
$G=(V,E)$. To this end, let $\kappa$ be a suitable upper bound on the cardinalities of the $\left|\mathcal{A}_i\right|$
and $S_i$ be the set of subsets of $\mathcal{A}$ of cardinality $i$ having a subspace distance of at least $d$. Setting
$S=\cup_{1\le i\le\kappa} S_i$ one can consider the optimization problem
\begin{align}
\max & \sum_{s \in S} \left|s\right|\cdot x_s \label{ILP_decompose}\\
& \sum_{s \in S} x_s = l \nonumber\\
&x_a + x_b \le 1 & \forall a \ne b \in S : a \cap b \ne \emptyset \nonumber\\
&x_s \in \{0,1\} & \forall s \in S \nonumber
\end{align}
for a given number $l$ of parts of the desired packing. Notwithstanding that the target function of ILP
formulation~(\ref{ILP_decompose}) completely ignores the correlation with the sizes of the items of the
second packing on $\Lambda$, it can be used to determine the exact value of $\Lambda$ in special cases,
see Subsection~\ref{subsec_i_10_6_4}. Setting \highlight{the vertex set of our graph $G$ to} $V=S$ and taking edges $e=\{s_1,s_2\}\in E$ iff $s_1\cap s_2\neq \emptyset$,
this corresponds to a vertex-weighted independent set problem with an additional restriction on
the number of chosen vertices. The algorithmic approaches described \highlight{in} \cite{kohnert2008construction} can be
adopted easily for \highlight{these} extra requirements.

Since the two subproblems from this subsection on their own even might be too hard, we may apply heuristic approaches
only. The very successful approach of prescribing automorphisms can also be applied here. Here the prescribed subgroup
of automorphisms has to be a subgroup of the automorphism group of $\mathcal{A}$ which typically is much smaller
than $\operatorname{GL}(n,q)$. However, {\lq\lq}good{\rq\rq} codes often have non-trivial automorphism groups.

\section{Examples}
\label{sec_examples}

In this section we describe the details of the coset construction for some parameters where we were
able to \highlightgreen{attain or} improve the best known constructions.
\highlightgreen{See Table~\ref{tab:improvedcodesizes} for an overview of code sizes where the Echelon-Ferrers construction from~\cite{etzion2009error} uses sizes of Echelon-Ferrers diagrams from~\cite{MR3480069} which was developed later.}

\begin{table}[H]
\begin{tabular}{p{3cm}p{3.5cm}l}
parameters & old largest known code & coset construction \\
\hline
$(8, \cdot, 4;4)_q$ & $q^{12}+\gaussm{4}{2}{q}(q^2+1)q^2+1$, cf.~\cite{etzion2013codes} & $q^{12}+\gaussm{4}{2}{q}(q^2+1)q^2+1$ \\
$(3k-3, \cdot, 2k-2;k)_q$ for $k \ge 4$ & \makecell[tl]{$q^{4k-6}+q^{k-1}+1$,\\cf.~\cite{MR3480069,etzion2009error}} & $q^{4k-6}+q^{k-1}+1$ \\
$(10,\cdot,6;4)_2$ & $4167$, cf.~\cite{etzion2009error} & $4173$ \\
\end{tabular}
\caption{\highlightgreen{Improved or attained code sizes by the coset construction.}}
\label{tab:improvedcodesizes}
\end{table}

\subsection{\texorpdfstring{$n=8$, $d=4$, $k=4$, and $q=2$ revisited}{n=8, d=4, k=4, and q=2 revisited}}
\label{subsec_8_4_4}
We apply the coset construction with $n'=4$, $k'=2$, $d'=2$ and use a parallelism in $\G{4}{2}{2}$ for
the $\mathcal{A}_i$ and $\mathcal{B}_i$. Here we have $l=7$ and $\left|\mathcal{A}_i\right|=\left|\mathcal{B}_i\right|=5$
for all $1\le i\le 7$. Thus, $\Lambda=7\cdot 5\cdot 5=175$.
\highlight{Since $\overline{F}$ is an MRD code of shape $2 \times 2$ and rank distance 2, its cardinality is 4, hence the}
corresponding code obtained from the coset construction has cardinality $700$. \highlight{The rank distance between the
two pivot vectors $v_1=(1,1,1,1,0,0,0,0)$ and $v_2=(0,0,0,0,1,1,1,1)$, as well as $v_i$ and
any pivot vector of any codeword of the coset construction for $i=1,2$ is 4, cf. Lemma~\ref{lemma_multilevel_coset_hamming}. So the Echelon-Ferrers construction applied
to $v_1$ and $v_2$ and combined with the coset construction yields a feasible subspace code for our parameters, i.e.,
$A_2(8,4;4)\ge 4096+700+1=4797$.} This is
\highlight{Construction~III} in \cite{etzion2013codes}. \highlight{A different technique was applied in \cite[Theorem~4.1]{Cossidente2016} to find a code of this size.}
Here, the MRD bound from Theorem~\ref{thm_MRD_upper_bound} is attained.
Recently, an $(8,4801,4;4)_2$ code has been found by a heuristic computer search \cite{heuristic_cliquer}.

As already observed in \cite{etzion2013codes}, the crucial ingredient for the feasibility of the above construction
is the existence of a parallelism in $\G{4}{2}{q}$. Performing the above cardinality computations for arbitrary
$q$ we obtain $A_q(8,4;4)\ge q^{12}+\gaussm{4}{2}{q}(q^2+1)q^2+1$, which also attains the MRD bound from
Theorem~\ref{thm_MRD_upper_bound}.

The authors of \cite{etzion2013codes} have remarked that they believe that their construction from their \highlight{Construction~III}
can be generalized to further parameters assuming the existence of a corresponding parallelism. This is indeed the case.

\begin{theorem}
  \label{thm_coset_parallelism}
  If $\mathcal{P}_1$ is a parallelism in $\G{n'}{k'}{q}$ and $\mathcal{P}_2$ a parallelism in $\G{n-n'}{k-k'}{q}$,
  then we can choose $\mathcal{A}=\mathcal{P}_1$, $\mathcal{B}=\mathcal{P}_2$, and $d=4$ in the coset construction.
  The corresponding code $\mathcal{C}$ attains the upper bound of Corollary~\ref{cor_upper_bound_Lambda}. If additionally
  $k-k'\ge 2$ and $n'-k'\ge 2$, then $\mathcal{C}$ is compatible with the lifted MRD code having pivot vector
  $(\underset{k}{\underbrace{1,\dots,1}},0,\dots,0)$.
\end{theorem}

\subsection{\texorpdfstring{$n=9$, $d=6$, $k=4$, and general field sizes $q$}{n=9, d=6, k=4, and general field sizes q}}
\label{subsec_i_9_6_4}

\highlightgreen{Since the combination of} the MRD code $\mathcal{C}_1$ with pivot vector $v=(1,1,1,1,0,0,0,0,0)$ and cardinality $1024$ with
the code $\mathcal{C}_2$ obtained from the explicit construction of Lemma~\ref{lemma_primal bound_Lambda} of
cardinality $8$
\highlightgreen{yields a $(9,1032,6;4)_2$ constant dimension code whose cardinality is one less than}
the MRD bound from Theorem~\ref{thm_MRD_upper_bound},
we were motivated to look for a coset construction yielding a larger addendum than $8$.

\begin{theorem}
  $A_q(9,6;4)\ge q^{10}+q^3+1$.
  \label{thm_improve_9_6_4}
\end{theorem}
\begin{proof}
  We choose $n'=4$, $k'=1$, and $d'=2$ in the coset construction. For the choice of $\mathcal{A}$ and
  $\mathcal{B}$ we observe $A_q(4,2;1)=q^3+q^2+q+1$ and $A_q(5,4;3)=A_q(5,4;2)=q^3+1$, see e.g.\
  \cite{beutelspacher1975partial}. Choose $\mathcal{A}$ and $\mathcal{B}$ as arbitrary codes attaining
  the mentioned upper bounds. Choosing a trivial packing of $\mathcal{B}$ into singletons yields a code $\mathcal{C}$
  of cardinality $q^3+1$. Adding the lifted MRD code of size $q^{10}$ gives the stated upper bound.
\end{proof}

We remark that the codes from Theorem~\ref{thm_improve_9_6_4} \highlight{meet} the MRD bound from Theorem~\ref{thm_MRD_upper_bound}.
The underlying construction can be generalized even more.

\begin{theorem}
  \label{thm_improve_series}
  For each $k\ge 4$ and arbitrary $q$ we have
  \[
    A_q(3k-3,2k-2;k)\ge q^{4k-6}+\highlightgreen{q^{k-1}}
    +1.
  \]
\end{theorem}
\begin{proof}
  We choose $n'=k$, $k'=1$, and $d'=2$ in the coset construction. For the choice of $\mathcal{A}$ and
  $\mathcal{B}$ we observe $A_q(k,2;1)=\gaussm{k}{1}{q}$ and
\highlightgreen{
  \begin{align*}
    &A_q(2k-3,2k-4;k-1)= A_q(2k-3,2k-4;k-2)\\
    &\overset{\text{\cite{beutelspacher1975partial}}}{=} \frac{q^{2k-3}-q}{q^{k-2}-1}-q+1=q^{k-1}+1<\highlight{\gaussm{k}{1}{q},}
  \end{align*}
}
  \highlight{where the first equality is true by considering the so-called complementary subspace code $C^{\perp}=\{U^{\perp} \mid U \in C\}$, cf. \cite{koetter2008coding}.}
   \highlight{Choose} $\mathcal{A}$ and $\mathcal{B}$ as arbitrary codes attaining
  the mentioned upper bounds. Choosing a trivial packing of $\mathcal{B}$ into singletons yields a code $\mathcal{C}$
  of cardinality \highlightgreen{$q^{k-1}+1$}. Adding a $(k\times (3k-3))$ lifted  MRD code gives the stated \highlight{lower} bound.
\end{proof}

We remark that the codes from Theorem~\ref{thm_improve_series} meet the MRD bound from Theorem~\ref{thm_MRD_upper_bound}.

\subsection{\texorpdfstring{$n=10$, $d=6$, $k=4$, and $q=2$}{n=10, d=6, k=4, and q=2}}
\label{subsec_i_10_6_4}
For the coset construction we choose $n'=4$ and $k'=1$. Since $\mathcal{A}\subseteq\G{4}{1}{2}$ we can only
have $D_S\left(\mathcal{A}_i\right)=2$, so that we must choose $d'=2$. Then, we can choose $\mathcal{A}=\G{4}{1}{2}$
and $\gaussm{4}{1}{2}=15$ singletons $\mathcal{A}_i$, which is obviously best possible. For $\mathcal{B}\subseteq
\G{6}{3}{2}$ we have the condition $D_S(\mathcal{B})\ge 4$. Reasonable candidates for $\mathcal{B}$ might be the
five isomorphism types of $(6,77,4;3)_2$ codes attaining the maximum cardinality $A_2(6,4;3)=77$, see \cite{hkk77}.
Using the first subproblem from Subsection~\ref{subsec_decomposition_cdc} we computationally obtain the upper bound
$\left|\mathcal{B}_i\right|\le 5=:\kappa$ for four out of the five isomorphism types. This information is enough
to conclude the upper bound $\Lambda(\mathcal{B})\le 15\cdot 5=75$. For the remaining isomorphism type,
i.e., the self-dual code having $168$ automorphisms which was labeled as {\lq\lq}type A{\rq\rq}, we
have $\left|\mathcal{B}_i\right|\le 7=:\kappa$. So, we solve the optimization problem~(\ref{ILP_decompose})
for $l=15$. The sizes of the requested sets $S_I$ are stated in Table~\ref{table_s_i}. The optimal target
value is $76$ and there exists a solution where the sizes of the elements in the packing are given by 4, 4, 4, 5,
5, 5, 5, 5, 5, 5, 5, 5, 5, 7, 7. Since in our situation we have $\left|\mathcal{A}_i\right|=1$ for all $i$, the
target function of (\ref{ILP_decompose}) coincides with the expression for $\Lambda$. Also the predefinition
of $l=15$ results in the maximum possible value, since we have $l\le 15$ from the $\mathcal{A}$-part and the
existence of a packing of $\mathcal{B}$ into $l'$ sets implies the existence of packings into $l\ge l'$ sets.
In general it is far from being obvious that we obtain the best possible codes from the coset construction
by choosing codes for $\mathcal{B}$ that have the maximal possible cardinality $A_q(n-n',d;k-k')$. However,
in our situation each choice for $\mathcal{B}$ different from the five considered isomorphism types of
$(6,77,4;3)_2$ codes has a cardinality of at most $76$, so that
$\sum_{i} \left|\mathcal{A}_i\right|\cdot \left|\mathcal{B}_i\right|\le 76$.

\begin{theorem}
  For $n=10$, $k=4$, $n'=6$, $k'=3$, $q=2$, and $d=6$, the maximum achievable $\Lambda$ of the coset construction
  is given by $76$.
\end{theorem}

\begin{table}[htp!]
\centering
\begin{tabular}{r|ccccccc}
$i=$&1&2&3&4&5&6&7 \\
\hline
$\left|S_i\right|=$&77&840&2240&1792&560&112&16\\[2mm]
\end{tabular}
\caption{Sizes of $S_i$ for $1\le i\le 7=\kappa$.}
\label{table_s_i}
\end{table}

For general field sizes $q$ we may choose $\mathcal{A}=\G{4}{1}{q}$ and $\gaussm{4}{1}{q}=q^3+q^2+q+1$ singletons $\mathcal{A}_i$.
For $\mathcal{B}$ one may choose a $(6,q^6+2q^2+2q+1,4;3)_q$ code, see \cite{hkk77}. Can one analytically describe
packings of $(6,q^6+2q^2+2q+1,4;3)_q$ codes into $q^3+q^2+q+1$ parts of large cardinality?

\begin{theorem}
 $A_2(10,6;4)\ge 4173$.
 \label{thm_improve_10_6_4}
\end{theorem}
\begin{proof}
  Let $\mathcal{C}_2$ be the code from the coset construction as outlined above. There is exactly
  one pivot vector $v=(1,1,1,1,0,0,0,0,0,0)$ satisfying the condition from Lemma~\ref{lemma_multilevel_coset_hamming}.
  The corresponding code $\mathcal{C}_1$ is the MRD code of size $\left\lceil 2^{6(4-3+1)}\right\rceil=4096$,
  so that $\left|\mathcal{C}_1\cup\mathcal{C}_2\right|=4172$. By a computer search we found a single codeword that can be
  added to $\mathcal{C}_1\cup\mathcal{C}_2$.
\end{proof}

We remark that the code from Theorem~\ref{thm_improve_10_6_4} meets the MRD bound from Theorem~\ref{thm_MRD_upper_bound}.
\highlight{By an exhaustive search we have verified that the general Echelon-Ferrers construction yields only codes with
$\left|\mathcal{C}\right|\le 4167$.}

\section{Conclusion}
\label{sec_conclusion}

The arguably most successful \highlight{generally} applicable construction for both constant dimension and subspace codes
of large minimum subspace distance is the Echelon-Ferrers construction from \cite{etzion2009error}. Here, we
have introduced a generalization of \highlight{\cite[Construction III]{etzion2013codes}}, which we call
\emph{coset construction}.
\highlightgreen{It turned out that the new construction is provably superior to the Echelon-Ferrers
construction for some parameters, see Subsection~\ref{subsec_i_10_6_4}.
We were able to apply the
coset construction to an infinite family of constant dimension codes that
attain the MRD bound from \highlight{\cite[Theorem 11]{etzion2013codes}}.}
So far all improvements include the usage of a lifted MRD code of maximal shape, so that these approaches
are all limited by the MRD bound from Theorem~\ref{thm_MRD_upper_bound}. For the relatively small
addendums constructed by the coset construction, we may utilize subcodes that have a larger cardinality
than the corresponding value of the MRD bound, see Subsection~\ref{subsec_i_10_6_4}. The constructions
of subspace codes based on the coset construction typically should yield many non-isomorphic codes, since
there are already many non-isomorphic MRD codes, see e.g.\ \cite{berger2003isometries,morrison2014equivalence}.
In Section~\ref{sec_optimal_parameters} we have obtained some first insights on the optimal choice of parameters
for the coset construction and related optimization problems. However, we are rather \highlight{far away} from
a clear assessment of the capabilities of the coset construction. This can be seen for example \highlight{through} the following
facts. \highlight{While} the coset construction is principally applicable for \highlightgreen{\emph{general}} subspace codes, we so far have
not found a single example improving one of the currently known lower bounds, \highlight{significantly contrasting the
situation for constant dimension codes}.

A more systematic analysis of {\lq\lq}good{\rq\rq} choices of parameters is needed. To
this end we propose some strongly related open research questions.
\highlight{As a benchmark, it would be very valuable
to generalize the MRD bound of Theorem~\ref{thm_MRD_upper_bound} to a larger class of parameters. As it is an open
problem whether the MRD bound of Theorem~\ref{thm_MRD_upper_bound} can be attained in all cases, it seems promising
to look at the corresponding open cases with more effort. Going along the lines of Theorem~\ref{thm_MRD_upper_bound}, one may
study upper bounds on $(n,M,d;k)_q$ codes that contain $(n,M',d';k)_q$ subcodes where $d'>d$, since such
results would give upper bounds on the achievable parameters $\kappa$, see Subsection~\ref{subsec_decomposition_cdc}.
This might give another hint which constant dimension codes may be appropriate for $\mathcal{A}$ and $\mathcal{B}$ within the
coset construction.}

\highlight{
Our analysis of the {\lq\lq}optimality{\rq\rq} of the example from Subsection~\ref{subsec_i_10_6_4} heavily relies on the
classification of $(6,77,4;3)_2$ codes. Since it possibly was only a matter of coincidence that we did not need to look
at codes of smaller cardinalities we would like to classify all codes attaining cardinality $A_q(n,d;k)$ up to
\highlight{isomorphism} extendability results, see e.g.\ \cite{nakic2015extendability}, at least for moderate parameters.
In order to generalize this example for field sizes $q>2$, packings of $(6,q^6+2q^2+2q+1,4;3)_q$ codes into $q^3+q^2+q+1$ parts of
large cardinality have to be studied.
}

\highlight{
The construction of Theorem~\ref{thm_improve_series} can easily
by generalized to parameters $n=n'+2k-3$, where $k\ge 3$ and $n'\ge 3$. For $k'=1$, $d'=2$,
we can choose $\mathcal{A}=\G{n'}{1}{q}$ and $\mathcal{B}$ as a (maximal) partial $(k-2)$-spread
$\mathcal{P}$ in $\mathbb{F}_q^{n-n'}$. Then, a packing of $\mathcal{P}$ into $\gaussm{n'}{1}{q}$ parts
is needed. For the parameters of Theorem~\ref{thm_improve_series} this packing trivially exists.
We remark that the maximum size of partial $\tilde{k}$-spreads in $\mathbb{F}_q^{\tilde{n}}$ is
known for $\tilde{n}\equiv 0,1\pmod{\tilde{k}}$ for arbitrary $q$, see e.g.\ \cite{beutelspacher1975partial},
and for $\tilde{n}\equiv 2\pmod{\tilde{k}}$ and $q=2$, see \cite{spreadsk3,partial_spreads_sascha}. So it seems
useful to study packings of the known best constructions for partial spreads into $\gaussm{m}{1}{q}$
parts of large cardinality for different values of $m$.}

\section*{Acknowledgement}
The authors would like to thank the editor and the anonymous  referees for their remarks
significantly improving the presentation of our results.

\section*{Appendix: \highlight{Proof of Lemma 4}}
\begin{proof}
For $U,V\in\G{n}{k}{q}$ \highlightgreen{and $\tau$ from~(\ref{eqn:tau}),} we have
\[
  d_S(U,V) =2(\dim(U+V)-k)=2\left(\rk\left(\begin{array}{c}U\\V\end{array}\right)-k\right).
\]
\highlight{In the case when} $A=A'$ and $B=B'$ we conclude
\begin{eqnarray*}
&& d_S\left(
\tau^{-1}\left(\begin{pmatrix}A&\varphi_B(F)\\0&B\end{pmatrix}\right),
\tau^{-1}\left(\begin{pmatrix}A&\varphi_B(F')\\0&B\end{pmatrix}\right)
\right)\\
&=& 2\left(
\rk\begin{pmatrix}A&\varphi_B(F)\\0&B\\A&\varphi_B(F')\\0&B\end{pmatrix}
-k\right)\\
&=& 2\left(\rk\begin{pmatrix}A&0\\0&\varphi_B(F')-\varphi_B(F)\\0&B\end{pmatrix}-k\right)\\
&=& 2\left(\rk(A) + \rk\begin{pmatrix}\varphi_B(F')-\varphi_B(F)\\B\end{pmatrix} -k\right)
\end{eqnarray*}
Since the pivot columns of $B$ in $\varphi_B(F')-\varphi_B(F)$ consists solely of zeros, we have
\begin{eqnarray*}
&& 2\left(\rk(A) + \rk\begin{pmatrix}\varphi_B(F')-\varphi_B(F)\\B\end{pmatrix} -k\right)\\
&=& 2(\rk(A)+\rk(\varphi_B(F')-\varphi_B(F))+\rk(B)-k)\\
&=&2(k'+\rk(F'-F)+k-k'-k)\\
&=&2 \rk(F'-F)= 2 d_R(F,F').
\end{eqnarray*}

For $A \ne A'$ or $B \ne B'$ we similarly conclude
\begin{eqnarray}
  && d_S\left(\tau^{-1}\left(\begin{pmatrix}A&\varphi_{B}(F)\\0&B\end{pmatrix}\right),
     \tau^{-1}\left(\begin{pmatrix}A'&\varphi_{B'}(F')\\0&B'\end{pmatrix}\right)\right)\nonumber\\
  &=& 2\left(\rk\begin{pmatrix}A&\varphi_{B}(F)\\0&B\\A'&\varphi_{B'}(F')\\0&B'\end{pmatrix}-k\right)\label{ie_dist_coset}\\
  &\ge& 2\left( \rk\begin{pmatrix}A\\A'\end{pmatrix} +
        \rk\begin{pmatrix}B\\B'\end{pmatrix}-k\right)\nonumber,
\end{eqnarray}
using the fact that $\rk\begin{pmatrix}X&Y\\0&Z\end{pmatrix} \ge \rk(X) + \rk(Z)$ with equality if $Y$ is zero
and swapping rows or columns, respectively, does not change the rank. We continue with
\begin{eqnarray*}
  && 2\left( \rk\begin{pmatrix}A\\A'\end{pmatrix} +
        \rk\begin{pmatrix}B\\B'\end{pmatrix}-k\right),\\
  &=& 2\left(\frac{d_S(A,A')}{2}+k'+\frac{d_S(B,B')}{2}+k-k'-k\right)\\
  &=& d_S(A,A')+d_S(B,B').
\end{eqnarray*}
\end{proof}

% \bibliographystyle{amsplain}
% \bibliography{paper_coset}

\end{document}